\documentclass[11pt]{amsart}
\usepackage{amscd}

\newtheorem{theorem}{Theorem}
\newtheorem{proposition}{Proposition}
\newtheorem{lemma}{Lemma}

\title{On closed embeddings of free topological algebras}
\author{Taras Banakh, Olena Hryniv}
\address{Ivan Franko National University of Lviv}
\email{tbanakh@yahoo.com}

\begin{document}
\begin{abstract} Let $\mathcal K$ be a complete quasivariety of completely regular
universal topological algebras of continuous signature $\mathcal E$ (which means that $\mathcal K$ is closed under taking subalgebras, Cartesian products, and
includes all completely regular topological $\mathcal E$-algebras
algebraically isomorphic to members of $\mathcal K$). For a topological space $X$
by $F(X)$ we denote the free universal $\mathcal E$-algebra over $X$ in the class
$\mathcal K$.
Using some extension properties of the Hartman-Mycielski
construction we prove that for a closed subspace $X$ of a metrizable
(more generally, stratifiable) space
$Y$ the induced homomorphism $F(X)\to F(Y)$ between the respective free universal
algebras is a closed topological embedding. This generalizes one
result of V.Uspenski\u\i\ \cite{11} concerning embeddings of free
topological groups.
\end{abstract}

\maketitle

One of important recent achievements in the theory of free
topological groups is a charming theorem by O.Sipacheva \cite{9}
asserting that the free topological group $F(X)$ of a subspace $X$
of a Tychonoff space $Y$ is a topological subgroup of $F(Y)$ if and
only if any continuous pseudometric on $X$ can be extended to a
continuous pseudometric on $Y$, see \cite{9}. The ``only if''
part of this theorem was proved earlier by V.~Pestov \cite{8}
while the ``if'' part was proved by V.V.~Uspenski\u\i\ \cite{11} for
the partial case of metrizable (or more generally, stratifiable)
$Y$. To prove their theorems both Uspenski\u\i\ and Sipacheva used a
rather cumbersome technique of pseudonorms on free topological
groups which makes their method inapplicable to studying some
other free objects.

In this paper using a categorial technique based on extension
properties of the Hartman-Mycielski construction we shall
generalize the Uspenski\u\i\ theorem and prove some general results
concerning embeddings of free universal algebras. In should be
mentioned that the Hartman-Mycielski construction has been
exploited in \cite{2} for proving certain results concerning
embeddings of free topological inverse semigroups.

Now we remind some notions of the topological theory of universal
algebras developed by M.M.~Choban and his collaborators, see
\cite{5}. Under a {\it continuous signature} we shall understand
a sequence $\mathcal E=(E_n)_{n\in\omega}$ of topological spaces. A
{\it universal topological algebra of continuous signature $\mathcal E$} or briefly a {\it topological $\mathcal E$-algebra} is a
topological space $X$ endowed with a sequence of continuous maps
$(e_n:E_n\times X^n\to X)_{n\in\omega}$ called {\it algebraic
operations} of $X$. A subset $A\subset X$ is called a {\it
subalgebra} of $X$ if $e_n(E_n\times A^n)\subset A$ for all
$n\in\omega$. Under {\it a homomorphism} between topological
${\mathcal E}$-algebras $(X,(e^X_n)_{n\in\omega})$ and
$(Y,(e^Y_n)_{n\in\omega})$ we understand a map $h:X\to Y$ such
that $$e_n^Y(c,h(x_1),\dots,h(x_n))=h(e^X_n(c,x_1,\dots,x_n))$$
for any $n\in\omega$, $c\in E_n$, and points $x_1,\dots,x_n\in X$. Two
topological ${\mathcal E}$-algebras $X$, $Y$ are {\it (algebraically)
isomorphic} if there is a bijective homomorphism $h:X\to Y$. If,
in addition, $h$ is a homeomorphism, then $X$ and $Y$ are
{\it topologically isomorphic}.

Under a {\it free universal algebra} of a topological space $X$ in
a class ${\mathcal K}$ of topological $\mathcal E$-algebras we understand
a pair $(F(X),i_X)$ consisting of a topological ${\mathcal E}$-algebra
$F(X)\in{\mathcal K}$ and a continuous map $i_X:X\to F(X)$ such that
for any continuous map $f:X\to K$ into a topological ${\mathcal E}$-algebra $K\in\mathcal K$ there is a unique continuous homomorphism
$h:F(X)\to K$ such that $f=h\circ i_X$.
It follows that for any
continuous map $f:X\to Y$ between topological spaces there is a
unique continuous homomorphism $F(f):F(X)\to F(Y)$ such that
$F(f)\circ i_X=i_Y\circ f$. Our aim in the paper is to find
conditions on $f$ guaranteeing that the homomorphism $F(f)$ is a
topological embedding.

According to \cite{5}, a free universal algebra $(F(X),i_X)$ of
a topological space $X$ exists (and is unique up to a topological
isomorphism) provided ${\mathcal K}$ is {\it a
quasivariety} which means that the class ${\mathcal K}$ is closed
under taking subalgebras and arbitrary Cartesian products. A
quasivariety ${\mathcal K}$ of topological ${\mathcal E}$-algebras is
called {\it a complete quasivariety} if any
completely regular topological ${\mathcal E}$-algebra $X$
algebraically isomorphic to a topological algebra $Y\in{\mathcal K}$
belongs to the class $\mathcal K$.

Finally we remind that a regular topological space $X$ is called
{\it stratifiable} if there exists a function $G$ which assigns to
each $n\in\omega$ and a closed subset $H\subset X$, an open set
$G(n,H)$ containing $H$ so that
$H=\bigcap_{n\in\omega}\overline{G(n,H)}$  and $G(n,K)\supset
G(n,H)$ for every closed subset $K\supset H$ and $n\in\omega$. It
is known that the class of stratifiable spaces includes all
metrizable spaces and is closed with respect to many countable
operations over topological spaces, see \cite{3}, \cite{6}.

Now we are able to formulate one of our main results.

\begin{theorem}\label{T1} Let $\mathcal K$ be a complete quasivariety of
completely regular topological $\mathcal E$-algebras of
continuous signature $\mathcal E$.
For any closed topological embedding $e:X\to Y$ between
stratifiable spaces the induced homomorphism $F(e):F(X)\to F(Y)$
between the corresponding free algebras is a closed topological
embedding.
\end{theorem}

In fact, Theorem~\ref{T1} follows from a more general result involving
the construction of Hartman and Mycielski. This
construction appeared in \cite{7} and was often exploited in
topological algebra, see \cite{4}. For a topological
space $X$ let $HM(X)$ be the set of all functions $f:[0;1) \to
X$ for which there exists a sequence $0=a_{0}<a_{1}<\dots<a_{n}=1$
such that $f$ is constant on each interval $[a_{i-1},a_{i})$,
$1\leq i\leq n$. A
neighborhood sub-base of the Hartman-Mycielski topology of
$HM(X)$ at an $f\in HM(X)$ consists of sets $N(a,b,V,\varepsilon)$,
where
\begin{itemize}
\item[1)] $0\leq a<b\leq 1$, $f$ is constant on $[a;b)$, $V$ is
    neighbourhood of $f(a)$ in $X$ and $\varepsilon>0$;
\item[2)] $g\in N(a,b,V,\varepsilon)$ means that $|\{t\in [a;b):
    g(t)\notin V\} |<\varepsilon$, where $|\cdot|$ denotes the Lebesgue
    measure on $[0,1)$.
\end{itemize}

If $X$ is a Hausdorff (Tychonoff) space, then so is the space
$HM(X)$, see \cite{7}, \cite{4}. The construction $HM$ is functorial in the sense
that for any continuous map $p:X\to Y$ between topological spaces the map
$HM(p):HM(X)\to HM(Y)$, $HM(p):f\mapsto p\circ f$, is continuous,
see \cite{7}, \cite{4}, \cite{10}.

The space $X$ can be identified with a subspace of
$HM(X)$ via the embedding $hm_X:X\to HM(X)$ assigning to each point $x$ the constant
function $hm_X(x):t\mapsto x$. This embedding $hm_X:X\to HM(X)$
is closed if $X$ is Hausdorff. It is easy to see
that for any continuous map $f:X\to Y$ we get
a commutative diagram:
$$
\begin{CD}
X@>f>> Y\\
@V{hm_X}VV @VV{hm_Y}V\\
HM(X)@>{HM(f)}>> HM(Y).
\end{CD}
$$

Our interest in the Hartman-Mycielski construction is stipulated
by the following important extension result proven in \cite{1}.

\begin{proposition}\label{P1} For a closed subspace $X$ of a
stratifiable space $Y$ there is a continuous map $r:Y\to HM(X)$
extending the embedding $hm_X:X\subset HM(X)$.
\end{proposition}

It will be convenient to call a subspace $X$ of a space $Y$ {\it
an $HM$-valued retract of $Y$} if there is a continuous map
$r:Y\to HM(X)$ extending the canonical embedding $hm_X:X\subset HM(X)$.
In these terms Proposition 1 asserts that each closed subspace of
a stratifiable space $Y$ is an HM-valued retract of $Y$.

As a set, the space $HM(X)$ can be thought as a subset of the
Cartesian power $X^{[0,1)}$. Moreover, if $(X,(e_n)_{n\in\omega})$
is a topological ${\mathcal E}$-algebra, then $HM(X)$ is a subalgebra of
$X^{[0,1)}$.
Let $$\{e^{HM}_n:E_n\times HM(X)^n\to HM(X)\}_{n\in\omega}$$ denote
the induced algebraic operations on $HM(X)$. That is,
$e_n^{HM}(c,f_1,\dots,f_n)(t)=e_n(c,f_1(t),\dots,f_n(t))$ for
$n\in\omega$, $(c,f_1,\dots,f_n)\in E_n\times HM(X)^n$, and $t\in[0,1)$.
It is
easy to verify that the continuity of the operation $e_n$ implies the
continuity of the operation $e_n^{HM}$ with respect to the
Hartman-Mycielski topology on $HM(X)$). Thus we get

\begin{proposition}\label{P2} If $(X,(e_n)_{n\in\omega})$ is a topological
${\mathcal E}$-algebra, then $(HM(X), (e_n^{HM})_{n\in\omega})$ is a
topological ${\mathcal E}$-algebra too. Moreover
the embedding $hm_X:X\to HM(X)$ is a homomorphism of
topological ${\mathcal E}$-algebras.
\end{proposition}

Since $HM(K)$ is
algebraically isomorphic to a subalgebra of $X^{[0,1)}$, we
conclude that for each completely regular topological ${\mathcal E}$-algebra $X$
belonging to a complete quasivariety $\mathcal K$
of topological ${\mathcal E}$-algebras the
${\mathcal E}$-algebra $HM(X)$ also belongs to the quasivariety $\mathcal K$.
Now we see that Theorem 1 follows from Proposition 1 and

\begin{theorem}\label{T2} Let ${\mathcal K}$ be a quasivariety
of (Hausdorff) topological ${\mathcal E}$-algebras of continuous
signature ${\mathcal E}$. Then for a subspace $X$ of a
topological space $Y$ the homomorphism $F(e):F(X)\to F(Y)$ induced
by the natural inclusion $e:X\to Y$ is a (closed) topological
embedding provided $X$ is an $HM$-valued retract of $Y$ and
$HM(F(X))\in\mathcal K$.
\end{theorem}

\begin{proof} Suppose that $HM(F(X))\in\mathcal K$ and $X$ is a $HM$-valued
retract of $Y$. The latter means that there is a
continuous map $r:Y\to HM(X)$ such that $hm_X=r\circ e$ where
$hm_X:X\to HM(X)$ and $e:X\to Y$ are natural embeddings. Applying
to the maps $hm_X$, $e$ and $r$ the functor $F$ of the free
universal ${\mathcal E}$-algebra in the quasivariety $\mathcal K$,
we get the equality $F(hm_X)=F(r)\circ F(e)$.

Applying the functor $HM$ to the canonical map $i_X:X\to F(X)$ of $X$ into its free
universal algebra, we get a continuous map $HM(i_X):HM(X)\to HM(F(X))$.
Taking into account that the ${\mathcal E}$-algebra $HM(F(X))$ belongs to the
quasivariety $\mathcal K$ and using the definition of the free algebra
$(F(HM(X)),i_{HM(X)})$ we find a unique continuous
homomorphism $h:F(HM(X))\to HM(F(X))$ such that
$h\circ i_{HM(X)}=HM(i_X)$. Let us show that $h\circ
F(hm_X)=hm_{F(X)}$. Since the maps $h\circ F(hm_X)$ and
$hm_{F(X)}$  are homomorphisms from the free algebra $F(X)$ of
$X$, to prove the equality $h\circ F(hm_X)=hm_{F(X)}$ it suffices
to verify that $h\circ F(hm_X)\circ i_X=hm_{F(X)}\circ i_X$.

By the definition of the homomorphism $F(hm_X)$ we get the commutative diagram
$$
\begin{CD}
X @>{i_X}>> F(X)\\
@V{hm_X}VV @ VV{F(hm_X)}V\\
HM(X)@>{i_{HM(X)}}>> F(HM(X))
\end{CD}
$$
which implies that
$h\circ F(hm_X)\circ i_X=h\circ i_{HM(X)}\circ hm_X=HM(i_X)\circ hm_X$
by the choice of the homomorphism~$h$.

On the other hand, by the naturality of the transformations
$\{hm_Z\}$ we get the commutative diagram
$$ \begin{CD}
X @>{\;\;i_X\;\;}>> F(X)\\
@V{hm_X}VV @ VV{hm_{F(X)}}V\\
HM(X)@ > {\;HM(i_X)\;}>> HM(F(X))
\end{CD}
$$
which implies that $HM(i_X)\circ hm_X=hm_{F(X)}\circ
i_X$. Thus $h\circ F(hm_X)\circ i_X=hm_{F(X)}\circ i_X$ which just yields
$hm_{F(X)}=h\circ F(hm_X)=h\circ F(r)\circ F(e)$. Observe that the
map $hm_{F(X)}$ is an embedding. Moreover, it is closed if $F(X)$
is Hausdorff (which happens if the quasivariety ${\mathcal K}$
consists of Hausdorff ${\mathcal E}$-algebras). Now the following
elementary lemma implies that $F(e)$ is a (closed) embedding.\end{proof}

\begin{lemma} Let $f:X\to Y$, $g:Y\to Z$ be continuous maps. If
$g\circ f:X\to Z$ is a (closed) topological embedding, then so is
the map $f$.
\end{lemma}


\begin{thebibliography}{}

\bibitem{1} {\it Banakh T., Bessaga C.} On linear operators extending
(pseudo)metrics // Bull. Polish Acad. Sci. Math. -- 2000. -- V.48, No.1. --
P.201--208.

\bibitem{2} {\it Banakh T., Guran I., Gutik O.} Free topological inverse
semigroups // Matem. Studii. 15:1 (2001) 23--43.

\bibitem{3} {\it Borges C.} On stratifiable spaces // Pacific J\.
Math\. -- 1966. -- V.17. -- P.1--16.

\bibitem{4} {\it Brown R., Morris S.A.} Embedding in contractible
or compact objects // Colloq. Math. -- 1978. -- V.37. -- P.213--222.

\bibitem{5} {\it Choban M.M.} Some topics in topological algebra
// Topology Appl. -- 1993. -- Vol.~54. -- P.~183--202.

\bibitem{6} {\it Gruenhage G.} Generalized metric spaces // in:
Handbook of Set-Theoretic Topology (K.Kunen and J.Vaughan eds.), Elsevier
Sci., 1984. -- P.423--501.

\bibitem{7} {\it Hartman S., Mycielski J.} On the imbeddings of
topological groups into connected topological groups
// Colloq. Math. -- 1958. -- Vol.~5. -- P.~167--169.

\bibitem{8} {\it Pestov V.G.} Some properties of free topological
groups // Vestnik Mosk. Gos. Univ. Ser. Mat. Mekh.
-- 1982. -- V.31. -- P.35--37 (in Russian).

\bibitem{9} {\it Sipacheva O.V.} Free topological groups of spaces
and their subspaces // Topology Appl. -- 2000. -- V.101. --
P.181--212.

\bibitem{10} {\it Teleiko A., Zarichny\u\i\ M.} Categorial topology
of compact Hausdorff spaces, VNTL, Lviv, 1999.

\bibitem{11} {\it Uspenski\u\i\  V.V.} Free topological groups of metrizable
spaces // Izv. Akad. Nauk SSSR. Ser. Mat. -- 1990. -- V.56, No.6.
-- P.1295--1319 (in Russian).
\end{thebibliography}
\end{document}